[1996/12/01]
\documentclass{article}
\usepackage{amsmath,amsthm}

\usepackage{mathtools}
\usepackage{hyperref}

\newtheorem{thm}{Theorem}[section]
\newtheorem{lem}[thm]{Lemma}

\let\ep=\epsilon

\let\vph=\varphi

\let\abs=\envert

\title{On the divisibility of odd perfect numbers, quasiperfect numbers and
amicable numbers by a high power of a prime
}

\date{}
\author{Tomohiro Yamada}
\begin{document}
\maketitle

\begin{abstract}
We shall give an explicit upper bound for the smallest prime factor of multiperfect numbers
of the form $N=p_1^{\alpha_1}\cdots p_s^{\alpha_s} q_1^{\beta_1}\cdots q_t^{\beta_t}$
with $\beta_1, \ldots, \beta_t$ bounded by a given constant.  We shall also give
similar results for quasiperfect numbers and relatively prime amicable pairs
of opposite parity.
\end{abstract}

\section{Introduction}\label{intro}

Let $\sigma(N)$ denote the sum of divisors of $N$ for a positive integer $N$ and
define $h(N)=\sigma(N)/N$.  An integer $N$ is said to be perfect if $h(N)=2$.
It is one of oldest and most infamous problems whether there exists any odd perfect number.
Moreover, it is also unknown whether there exists any odd integer $N$ with $h(N)=k$
for some integer $k>1$.

Although it is unknown whether there exists any odd perfect number,
it is known that an odd perfect number must satisfy various conditions.
Suppose that $N$ is an odd perfect number.
Euler has shown that $N=p^{\alpha} q_1^{\beta_1}\cdots q_t^{\beta_t}$, where
$p, q_1, \ldots, q_t$ are distinct odd primes with $p\equiv \alpha\equiv 1\pmod{4}$
and $\beta_1, \ldots, \beta_t$ even.
Steuerwald \cite{St} proved that we cannot have $\beta_1=\cdots =\beta_t=2$.
If $\beta_1=\cdots =\beta_t=\beta$, then it is known that
$\beta\neq 4$ (Kanold \cite{Ka1}), $\beta\neq 6$ (Hagis and McDaniel \cite{HM}),
$\beta\neq 10, 24, 34, 48, 124$ (McDaniel and Hagis \cite{MDH}), $\beta\neq 12, 16, 22, 28, 36$ (Cohen and Williams \cite{CoW}).
In their paper \cite{MDH}, Hagis and McDaniel conjecture that
$\beta_1=\cdots =\beta_t=\beta$ does not occur.
The author \cite{Ymd1} proved that there are only finitely many odd perfect numbers
for any given $\beta$.
McDaniel \cite{Mc1} proved that we cannot have $\beta_1\equiv\cdots \equiv\beta_t\equiv 2\pmod{6}$,
i.e., $3$ cannot divide all of $\beta_1+1, \beta_2+1, \ldots, \beta_t+1$.
If $m$ divides all of $\beta_1+1, \beta_2+1, \ldots, \beta_t+1$,
then it is known that $m\neq 35$ (Hagis and McDaniel \cite{MDH})
and $m\neq 65$ (Evans and Pearlman \cite{EP}) and eventually
Fletcher, Nielsen and Ochem \cite{FNO} showed that $m\neq 5$ as a by-product
of their main result, which will be discussed later.
In general, if a prime $l$ divides all of $\beta_1+1, \beta_2+1, \ldots, \beta_t+1$,
then $l^4$ must divide $N$ by a result of Kanold \cite{Ka1}.

However, if we relax the condition that there exists some integer
dividing all of $\beta_1+1, \beta_2+1, \ldots, \beta_t+1$,
then the situation becomes quite different.
The simplest problem in this direction would be whether there exists
an odd perfect number of the form $p^{\alpha}q_1^{\beta_1}q_2^{\beta_2}\cdots q_t^{\beta_t}$
with $p\equiv \alpha\equiv 1\pmod{4}$ and $\beta_i\leq 4$.
This problem has been studied by McDaniel \cite{Mc2} and Cohen \cite{Co2}.
These papers give {\it lower} bounds for the smallest prime factor of $N$:
the first paper shows it must be at least $101$ and the second shows it must be at least $739$.

In general, we can make a conjecture that for a fixed finite
set $\mathcal{P}$ of integers, a fixed rational number $n/d$ and a fixed integer $s$,
there exist only finitely many odd $n/d$-perfect numbers
$N=p_1^{\alpha_1}\cdots p_s^{\alpha_s} q_1^{\beta_1}\cdots q_t^{\beta_t}$ with $\beta_1+1, \ldots, \beta_t+1$ contained in $\mathcal{P}$.

This conjecture still seems to be far beyond reach,
though this conjecture is weaker than the finiteness conjecture
of odd $n/d$-perfect numbers.
In the preprint \cite{Ymd2}, using sieve methods, the author has proved that
for a fixed finite set $\mathcal{P}$ of integers, a fixed rational number $n/d$ and a fixed integer $s$,
there exists an effective constant $C$ such that odd $n/d$-perfect numbers of the form
$N=p_1^{\alpha_1}\cdots p_s^{\alpha_s} q_1^{\beta_1}\cdots q_t^{\beta_t}$ with $\beta_+11, \ldots, \beta_t+1$ contained in $\mathcal{P}$ must have a prime divisor smaller than $C$.
Moreover, the author has proved that, in the case $N$ is perfect and $\beta_i\leq 4$, then
$C$ can be taken to be $\exp(4.97401\times 10^{10})$.

Using the author's method, but with the aid of the large sieve instead of
Selberg's sieve used by the author \cite{Ymd1}, Fletcher, Nielsen and Ochem \cite{FNO} proved that
if $N=p_1^{\alpha_1}\cdots p_s^{\alpha_s} q_1^{\beta_1}\cdots q_t^{\beta_t}$
satisfies $h(N)=n/d$ and for each $i$, $\beta_i+1$ has a prime factor belonging to
a finite set $\mathcal{P}$ of primes, then $N$ has a prime divisor small than a effective constant
$C$, depending only on $n, s$ and $\mathcal{P}$.
Moreover, they proved that the smallest prime factor of an odd perfect number
$N$ satisfying the above condition with $\mathcal{P}=\{3, 5\}$
lies between $10^8$ and $10^{1000}$, improving results in \cite{Co2} and \cite{Ymd2}.
This implies that neither $\mathcal{P}=\{3\}$ nor $\mathcal{P}=\{5\}$ can occur
since a prime $l$ must divide $N$ if $\mathcal{P}=\{l\}$ by the result of Kanold \cite{Ka1}
mentioned above.

However, they did not give an explicit value for their effective $C$ in other cases.
In this paper, the author would like to give an explicit upper bound for $C$ in general cases.
\begin{thm}\label{thm1}
Let $\mathcal{P}$ be a finite but nonempty set of primes
and $n, d, \beta_1, \ldots, \beta_t$ be positive integers such that for each $i=1, \ldots , t$,
$\beta_i+1$ is divisible by at least one prime in the set $\mathcal{P}$
and let $P$ denote the product $\prod_{p\in\mathcal{P}}p$.
Define $\Omega_\mathcal{P}(x)$ to be the number of prime factors of $x$ that belong to $\mathcal{P}$,
counting multiplicity and let $s_0=s+\omega(n)+\Omega_\mathcal{P}(n)$.
Furthermore, let $L(\ep, n)$ be the real number $x$ such that $\Omega(n)=\ep x/(\log^2 x)$,
\begin{equation}
x_1=x_1(l)=x_1(s_0; l, P)=\max\{\exp P, \exp (100.7l), \exp(\exp(9)), 10s_0(l-1)+1\}
\end{equation}
for each prime $l$ in $\mathcal{P}$ and, for any $\ep>0$, $C_0=C_0(d, s, n, P, \ep)$ be the maximum among quantities $2(d+1)s, x_1(l)^{8.35}, L(\ep, n)$ and
\begin{equation}
\ep+\exp \left(\frac{(17.62196\vph(P)+129.5214(l-1))\abs{\mathcal{P}}\log x_1}{(l-1)\log\frac{n}{d}}\right)
\end{equation}
with $l$ running over all primes in $\mathcal{P}$.

If $N=p_1^{\alpha_1}\cdots p_s^{\alpha_s} q_1^{\beta_1}\cdots q_t^{\beta_t}$
satisfies $h(N)=\frac{n}{d}$, then, for any $\ep>0$, $N$ has a prime factor smaller than $C_0$.
\end{thm}

For fixed $s$ and $n$, our upper bound is the order of exponential of $P\vph(P)\abs{\mathcal{P}}$,
rather than double-exponential of $\vph(P)\log P$ as in Theorem 3 of \cite{FNO}.

We note that no absolute upper bound is known
for the smallest prime factor of a {\em general} odd perfect number
if it exists at all; another known result is Gr{\"u}n's result \cite{Gru}
that the smallest prime factor must be smaller than $\frac{2}{3}\omega(N)+2$,
where $\omega(N)$ denotes the number of distinct prime factors of $N$.

We shall also give a few more applications of sieve methods to divisor-related numbers.
Cattaneo \cite{Cat} called a positive integer $N$ quasiperfect if $\sigma(N)=2N+1$
and showed that such an integer must be an odd square and any divisor of $\sigma(N)$
must be congruent to $1$ or $3$ modulo $8$.
Hagis and Cohen \cite{HC} showed that if $N$ is quasiperfect, then
$N>10^{35}$ and $N$ has at least $7$ distinct prime factors.

Cohen \cite{Co1} showed that if $p_1, p_2, \ldots, p_t$ are distinct primes and $(p_1 p_2 \cdots p_t)^{2a}$ is
quasiperfect, $a$ must be congruent to $1, 3, 5, 9$ or $11 \pmod{12}$.
Moreover, if an integer of the form $p_1^{6a_1+2}p_2^{6a_2+2}\cdots p_t^{6a_t+2}$
is quasiperfect, then $t\geq 230876$.  We shall show the following analogue of
Theorem \ref{thm1}.

\begin{thm}\label{thm2}
Let $\mathcal{P}$ be a finite set of primes
and $\alpha_1, \alpha_2, \ldots, \alpha_t$ be positive integers such that for each $i=1, \ldots , t$,
$\alpha_i+1$ is divisible by at least one prime in the set $\mathcal{P}$.
If $N=p_1^{2\alpha_1}p_2^{2\alpha_2}\cdots p_t^{2\alpha_t}$ is quasiperfect,
then $N$ must have a prime factor smaller than an effectively computable constant
$C_1$ depending only on $\mathcal{P}$,
which can be made explicit as follows:
\begin{equation}
x_3=x_3(l)=\max\{\exp (8l), \exp(\exp(9))\}, C_1=\max_{l\in \mathcal{P}} x_3^{2310\abs{\mathcal{P}}^2}.
\end{equation}
\end{thm}

Our method can also be applied to special amicable pairs.
A pair of integers $m, n$ are called amicable if the two equations
$n=\sigma(m)-m$ and $m=\sigma(n)-n$ hold simultaneously or,
equivalently, $\sigma(m)=\sigma(n)=m+n$.
It is unknown whether there exists a relatively prime amicable pair
or even whether there exists an amicable pair of opposite parity.

Assume that $m$ is even, $n$ is odd and $m, n$ are relatively prime amicable numbers.
Kanold showed that $(m, n)=(2M^2, N^2)$ for some odd integers $M, N$ in \cite{Ka2}
and that $mn$ must have at least $21$ distinct prime factors in \cite{Ka3}.
Hagis \cite{Ha1} showed that $mn$ cannot be a multiple of $3$ and $mn\geq 10^{74}$.
Moreover, if $5$ does not divide $mn$, then $mn\geq 10^{238}$ and $mn$ must have
at least $53$ distinct prime factors.
Later Hagis showed that both $m, n>10^{60}$ and $mn>10^{121}$ in \cite{Ha2}
and that $mn$ must have at least $22$ distinct prime factors in \cite{Ha3}.

Under a slightly more general condition that $m, n$ are relatively prime and
$\sigma(m)\sigma(n)=(m+n)^2$, Kishore \cite{Kis} showed that $4$ does not divide $mn$ and $mn$ must have
at least $22$ distinct prime factors.

We have the following analogue of Theorem \ref{thm1}.
\begin{thm}\label{thm3}
Let $\mathcal{P}$ be a finite set of primes
and $\beta_1, \beta_2, \ldots, \beta_t$ be positive integers such that for each $i=1, \ldots , t$,
$2\beta_i+1$ is divisible by at least one prime in $\mathcal{P}$.
If $m$ is even, $n$ is odd and $m, n$ are relatively prime integers satisfying
$\sigma(m)\sigma(n)=(m+n)^2$ and $mn=2^\alpha p_1^{2\beta_1}p_2^{2\beta_2}\cdots p_t^{2\beta_t}$,
then $mn$ must have a prime factor less than $C_1$, where $C_1$ is the same as in the previous theorem.
\end{thm}

Indeed, both Theorems \ref{thm2} and \ref{thm3} follow from
the following general result.

\begin{thm}\label{thm4}
Let $\mathcal{P}$ be a finite set of primes.
If $N=p_1^{2\beta_1}p_2^{2\beta_2}\cdots p_t^{2\beta_t}$, with each $2\beta_i+1$
divisible by some prime in $\mathcal{P}$, is an odd integer such that
$\sigma(N)/N\geq 2$ and $\sigma(N)$ has no prime factor congruent to $5$ or $7$ modulo $8$,
then $N$ must have a prime factor smaller than $C_1$.
\end{thm}

For quasiperfect numbers of the form $(p_1p_2\cdots p_t)^{2\beta}$, we obtain
stronger results.  In \cite{Ymd4} we showed that if $N=(p_1p_2\cdots p_t)^{2\beta}$
with $p_1<p_2<\cdots <p_t$ is quasiperfect, then $2\beta+1$ must be divisible by $3$ and
$p_1<\exp(721.85)<3.129477 \cdot 10^{313}$.
This upper bound is still considerably large and we cannot even prove that
$p_1>7$.

\section{Upper bound sieve}

Our main tool is a standard result in large sieve theory.
However, for convenience to compute explicit bounds, we must use an explicit (but a little sophisticated) upper bound sieve formula.
There are several explicit upper bound sieve formulae to obtain explicit upper bound for the implied constant in an upper bound sieve.
In \cite{Ymd2}, the author used the upper bound formula following from Selberg's sieve.
But here we shall use the large sieve formula used by Fletcher, Nielsen and Ochem \cite{FNO}, which enabled them
to obtain a considerably stronger estimate than in the author's paper \cite{Ymd2}.

Firstly, we would like to introduce some notations.  Let $X$ be a positive number
and $A$ be a set of integers contained in an interval of length at most $X$.
For each prime $p$, let $\Omega_p$ be a set of residue classes modulo $p$
and $\rho(p)$ denote the number of residue classes in $\Omega_p$.
Define $P(z)=\prod_{p<z}p$ to be the product of primes less than $z$,
$g(m)$ to be the multiplicative function over the squarefree integers $m$
with $g(p)=\rho(p)/(p-\rho(p))$ for each prime $p$,
\begin{equation*}
V(Q)=\prod_{p\mid Q}\left(1-\frac{\rho(p)}{p}\right)
\end{equation*}
for any real $Q$, where $p$ runs over primes, and
\begin{equation*}
G_z(T)=\sum_{d\leq T, d\mid P(z)}g(d), G(T)=G_T(T).
\end{equation*}
Finally, we define $S(A, z)=S(A, z, \Omega)$ to be the number of
integers in $A$ that do not belong to $\Omega_p$ for any prime $p$ dividing $P$.

Now we introduce two lemmas concerning the large sieve inequality.
These inequalities allow us to calculate an upper bound in Theorem \ref{thm1} explicitly.
\begin{lem}\label{lm21}
Assume that $\rho(p)<p$ for any prime $p$.
Then it holds for any $w\geq 1$ that
\begin{equation}
S(A, w)\leq \frac{X+w^2}{G(w)}.
\end{equation}
\end{lem}
\begin{proof}
It immediately follows from Theorem 7.14 in \cite{IK} applied with
$\Omega_p$ restricted to primes $p<z$ and
$h(m)$ the multiplicative function over squarefree integers $m$ defined by
\[h(p)=
\begin{cases}
g(p) & \textrm{for all primes } p<z. \\
0 & \textrm{for other primes.}
\end{cases}
\]
\end{proof}

\begin{lem}\label{lm22}
Let us denote
\begin{equation}
B(z)=\frac{1}{\log{z}}\sum_{p<z}\frac{\rho(p)\log{p}}{p}
\end{equation}
and
\begin{equation}
\psi_1(K, t)=\max\left\{0, t\log{\frac{t}{K}}-t+K\right\}.
\end{equation}

If $z\geq 2$ and $v=(\log x)/(\log z)\geq u B(z)$, then we have
\begin{equation}
G_z(x^{1/u})\geq \frac{\psi_0(v, u)}{V(P(z))},
\end{equation}
where
\begin{equation}
\psi_0(v, u)=1-\exp(-\psi_1(B(z), v/u)).
\end{equation}
\end{lem}
\begin{proof}
This is Theorem 2.2.1 in \cite{Gre} if we take $B=\sup_t B(t)$ instead of $B(z)$.
But we can see that this theorem still holds with $B(z)$ in place of $B$
whether the supremum $B$ exists or not.
Indeed, it follows from the argument on pages 53--54 in \cite{Gre} that
\begin{equation}
1-V(P(z))G_z(x^{1/u})\leq \exp\left(-c\frac{\log x}{u\log z}+B(z)(e^c-1)\right)
\end{equation}
for any constant $c\geq 0$.  Setting $c=\log (v/u)-\log B(z)$, we obtain
the lemma.
\end{proof}

\section[Proof of Theorem 1.1]{Proof of Theorem \ref{thm1}}

In this section, we shall give a proof of Theorem \ref{thm1} without making constants explicit.
Explicit constants shall be given in the next section.

We may assume that $P\geq 21$ by virtue of the result in \cite{FNO} concerning the case
$\mathcal{P}=\{3, 5\}$ mentioned in the introduction of this paper.
Let $N=p_1^{\alpha_1}\cdots p_s^{\alpha_s} q_1^{\beta_1}\cdots q_t^{\beta_t}$ be a solution of $h(N)=\frac{n}{d}$.
Let us denote by $T$ the set of primes $\equiv 1\pmod{P}$
and by $T_y$ the set of primes congruent to $1\pmod{P}$ or congruent to $1\pmod{l}$ and not exceeding $y$.
If $N$ has a prime divisor in $\mathcal{P}$, then clearly $N$ has a prime factor smaller than $C_0$.
We may assume without loss of generality that $N$ has no prime divisor in $\mathcal{P}$
and therefore $\Omega_\mathcal{P}(N)=0$.

Let $Q_l$ denote by the set of primes $q_i$ with $\beta_i+1$ divisible by $l$
and $\pi_l(x)$ denote the number of primes not exceeding $x$ that belong to $Q_l$.
By assumption, any $q_i$ belongs to $Q_l$ for some $l$ in $\mathcal{P}$.

Now we shall prove a result concerning the distribution of prime factors of $N$,
which is the most important lemma in the proof of Theorem \ref{thm1}.
\begin{lem}\label{lm3}
Choose any $l$ from $\mathcal{P}$.
Let $\kappa=\frac{l-1}{\vph(P)}$ and $y$ be a sufficiently large real number.
There exist three constants $B_0, B_1$ and $X_1=X_1(s_0; l, P)$ depending only on $s_0, l$ and $P$,
which shall be made explicit later,
such that if $u, v$ and $y$ are real numbers with $u\geq 2, v>B_0 u$ and $y\geq X_1$
and $N$ has no prime factor $\leq y$, then we have
\begin{equation}
\pi_l(x)\leq \Omega(n)+
\begin{dcases}
\frac{B_1(1+x^{2/u-1}) v^2 x}{\xi(v, u) \log^2 x} & \text{for } \max\{y, X_1^v\} \leq x<y^v, \\
\frac{B_1(1+x^{2/u-1}) v^{1+\kappa} x}{\xi(v, u) \log^{1-\kappa} y \log^{1+\kappa} x} & \text{for } x\geq y^v,
\end{dcases}
\end{equation}
where $\xi(v, u)=\psi_0(B_0, v/u)$.
\end{lem}
\begin{proof}
Let $\pi^*_l(x)$ denote the number of primes $q_i\leq x$ that belong to $Q_l$ such that
$\sigma(q_i^{l-1})$ has no common prime factor smaller than $X_1$ with $n$.
By assumption, $N$ has no prime factor smaller than $X_1$
and therefore there exist at most $\Omega(n)$ prime factors $q_i$
such that $\sigma(q_i^{l-1})$ has any common prime factor smaller than $X_1$ with $n$.
This immediately gives that
\begin{equation}\label{eq30}
\pi_l(x)<\pi^*_l(x)+\Omega(n).
\end{equation}

Now, let $U=U_l$ be the set of primes congruent to $1\pmod{P}$ or congruent to $1\pmod{l}$ and
not exceeding $y$ except primes dividing $N$ or primes above or equal to $X_1$ dividing $n$.
Namely, we set $U_l=T_y\backslash (\mathrm{pf}(N)\cup (\mathrm{pf}(n)\cap [X_1, \infty)))$,
where $\mathrm{pf}(m)$ denotes the set of prime factors of an integer $m$.
So that, if a prime divisor $r$ of $nN$ belongs to $U$, then $r$ divides $n$ and $r\geq X_1$.
Hence, we see that if $q_i\in Q_l$ and $\sigma(q_i^{l-1})$ has no common prime factor
smaller than $X_1$ with $n$, then $\sigma(q_i^{l-1})$ is divisible by no prime in $U$.

Let $r$ be a prime in $U$.  Then, since $r\equiv 1\pmod{l}$, there are exactly $l-1$
congruence classes $g_1(r), \ldots, g_{l-1}(r) \pmod{r}$ that belong to order $l$.
Since $r$ does not divide $\sigma(q_i^{l-1})$, $q_i$ belongs to none of the
$l$ classes $0, g_1, \ldots, g_{l-1} \pmod{r}$.

Now we can apply the sieve method described in the previous section with $A$ the set of integers
not exceeding $x$, $X=x$, $\Omega^{(l)}_r$ the set of integers not exceeding $x$
that belong to any of congruence classes $0, g_1, \ldots, g_{l-1}\pmod{r}$
for $r\in U$ and $0\pmod{r}$ for $r\not\in U$, $\rho(r)=l$ for $r\in U$ and $\rho(r)=1$ for $r\not\in U$.
Thus we see that if $q$ is a prime greater than $x^{1/u}$ in $Q_l$ counted by $\pi^*_l$,
then $q$ belongs to none of the congruence classes $\Omega^{(l)}_r$ with $r\leq x^{1/u}$.
Hence, letting $A$ the set of integers not exceeding $x$, we have, 
\begin{equation}
\pi^*_l(x)\leq S(A, x^{1/u}, \Omega^{(l)})+x^{1/u}
\end{equation}
and, using (\ref{eq30}),
\begin{equation}
\pi_l(x)\leq S(A, x^{1/u}, \Omega^{(l)})+x^{1/u}+\Omega(n).
\end{equation}
We can easily see that $\rho(r)<r$ for any prime $r$ and,
provided that $v/u\geq B(z)$, Lemmas \ref{lm21} and \ref{lm22} with $w=x^{1/u}$ give
\begin{equation}\label{eq31}
\pi_l(x)\leq \frac{x+x^{2/u}}{G(x^{1/u})}+x^{1/u}\leq \frac{x(1+x^{2/u-1})V(P(z))}{\psi_0(v, u)}+x^{1/u}+\Omega(n),
\end{equation}
where we put $z=x^{1/v}$, observing that $G_w(w)\geq G_z(w)$ for $w\geq z$.

Now we need to confirm that $v/u\geq B(z)$ and obtain an upper bound for the quantity $V(P(z))/\psi_0(v, u)$.
There are two cases: $x\geq y^v$, i.e. $z\geq y$ and $x<y^v$, i.e. $z<y$.
In both cases, we shall obtain an upper bound for $B(z)$ and then $V(P(z))$.

We begin by considering the case $z\geq y$.
We see that
\begin{equation}
\begin{split}
& \sum_{p\leq z}\frac{\rho(p)\log p}{p} \\
& \leq \sum_{p\leq z}\frac{\log p}{p}+\sum_{\substack{p\leq y,\\ p\equiv 1\pmod{l}}}\frac{(l-1)\log p}{p}+\sum_{\substack{y<p\leq z,\\ r\equiv 1\pmod{P}}}\frac{(l-1)\log p}{p}.
\end{split}
\end{equation}

From the theory of the distribution of primes in arithmetic progressions we see that
\begin{equation}\label{eq32a}
\sum_{p\leq y, p\equiv a\pmod{l}}\frac{\log p}{p}<\frac{\log y+A_1}{l-1}
\end{equation}
and
\begin{equation}\label{eq32b}
\sum_{y<p\leq z, p\equiv a\pmod{P}}\frac{\log p}{p}<\frac{1}{\vph(P)}\left(\log\frac{z}{y}+\frac{A_2}{\log y}\right)
\end{equation}
for some constants $A_1$ and $A_2$ if $y$ and $z$ are sufficiently large.
Hence, using the estimate $\sum_{p\leq z}(\log p)/p<\log z$ in \cite[(3.24), p. 70]{RS}, we obtain
\begin{equation}
\begin{split}
\sum_{p\leq z}\frac{\rho(p)\log p}{p}
\leq & \log z+\log y+A_1+\kappa(\log z -\log y)+\frac{A_2\kappa}{\log y} \\
\leq & (1+\kappa)\log z+(1-\kappa)\log y+A_1+\frac{A_2\kappa}{\log y},
\end{split}
\end{equation}
which is at most $B_0\log z$ recalling that $z\geq y\geq X_1$ now.
In other words, we have
\begin{equation}\label{eq32c}
B(z)<B_0.
\end{equation}
Hence, the assumption $v>B_0 u$ implies that $v/u>B(z)$.

Nextly, we shall obtain an upper bound for $V(P(z))$.
There can be at most $\Omega_\mathcal{P}(nN)=\Omega_\mathcal{P}(n)$
prime factors $q_i$ in $T$ since if $q_i\in T$, then $\sigma(q_i^{\beta_i})$ must be divisible by $\beta_i+1$
and therefore by some $l$ in $\mathcal{P}$.
Hence, there exist at most $s+\omega_\mathcal{P}(n)$ prime factors of $N$ in $T$,
which must be larger than $y\geq X_1$ since $N$ is assumed to have no prime factor not exceeding $y$.
Moreover, if $r<y$ is a prime $\equiv 1\pmod{l}$ which does not belong to $U$,
then $r$ must divide $n$ and therefore $r\geq X_1$.

Thus we conclude that $U$ consists of all primes in $T_y$
except at most $s_0=s+\omega(n)+\Omega_\mathcal{P}(n)$ primes, which are larger than $X_1$.
Hence, we obtain
\begin{equation}
\begin{split}
\prod_{r<z, r\in U}\left(1-\frac{1}{r}\right)
\leq & \prod_{\substack{r<z,\\ r\equiv 1\pmod{l},\\ r\not\in U}}\frac{r}{r-1}\prod_{\substack{X_1\leq r<z,\\ r\in U}}\left(1-\frac{1}{r}\right) \\
\leq & \left(1+\frac{1}{X_1-1}\right)^{s_0}\prod_{\substack{X_1\leq r<z,\\ r\in U}}\left(1-\frac{1}{r}\right) \\
<& \exp\frac{s_0}{X_1-1}\prod_{\substack{X_1\leq r<y,\\ r\equiv 1\pmod{l}}}\left(1-\frac{1}{r}\right)
\prod_{\substack{y\leq r<z,\\ r\equiv 1\pmod{P}}}\left(1-\frac{1}{r}\right).
\end{split}
\end{equation}
We see that if $k\geq 1$ and $Y, Z$ with $Z\geq Y$ are sufficiently large compared to $k$, then
\begin{equation}\label{eq32d}
\prod_{Y\leq p<Z, p\equiv 1\pmod{k}}\left(1-\frac{1}{p}\right)<\left(\frac{\log Y}{\log Z}\right)^{1/\vph(k)} \exp \left(\frac{A_3}{\vph(k)\log^2 Y}\right)
\end{equation}
for some constant $A_3$.  Since $z\geq y\geq X_1=X_1(s_0; l, P)$, we can apply (\ref{eq32d}) with $k=l$ and $k=P$ and obtain
\begin{equation}
\begin{split}
\prod_{r<z, r\in U}\left(1-\frac{1}{r}\right)
<& \left(\frac{\log X_1}{\log y}\right)^{1/(l-1)} \left(\frac{\log y}{\log z}\right)^{1/\vph(P)}\\
& \times\exp \left(\frac{s_0}{X_1-1}+\frac{A_3}{(l-1)\log^2 X_1}+\frac{A_3}{\vph(P)\log^2 y}\right).
\end{split}
\end{equation}
For $k=1$, an explicit formula of Mertens has been obtained in the form
$\prod_{p<z} (1-1/p)<e^{-\gamma}\log^{-1}z(1+1/(2\log^2 z))$ by \cite[(3.26), p. 70]{RS}.
Hence,
\begin{equation}
\begin{split}
V(P(z))= & \prod_{r<z}\left(1-\frac{\rho(r)}{r}\right) \leq \prod_{r<z}\left(1-\frac{1}{r}\right)^{\rho(r)} \\
= & \prod_{r<z}\left(1-\frac{1}{r}\right)\prod_{r<z, r\in U}\left(1-\frac{1}{r}\right)^{l-1} \\
< & \frac{e^{-\gamma}\log X_1}{\log^{1-\kappa} y \log^{1+\kappa} z}\left(1+\frac{1}{2\log^2 z}\right) \\
& \times \exp\left(\frac{s_0(l-1)}{X_1-1}+\frac{A_3}{\log^2 X_1}+\frac{A_3\kappa}{\log^2 y}\right).
\end{split}
\end{equation}
Provided that $X_1$ is sufficiently large compared to $s_0$ and $l$,
we have
\begin{equation}\label{eq33a}
V(P(z))<\frac{A_4\log X_1}{\log^{1-\kappa} y \log^{1+\kappa} z}
\end{equation}
for some constant $A_4$.
Since $B(z)<B_0\leq v/u$ by (\ref{eq32c}), we have $\psi_0(v, u)=1-\exp(-\psi_1(B(z), v/u))>1-\exp(-\psi_1(B_0, v/u))=\xi(v, u)$ and therefore
\begin{equation}\label{eq34a}
\frac{V(P(z))}{\psi_0(v, u)}\leq \frac{A_4\log X_1 (1+x^{2/u-1}) v^{1+\kappa}}{\psi_0(v, u)\log^{1-\kappa} y\log^{1+\kappa} x}
\leq \frac{A_4\log X_1 (1+x^{2/u-1}) v^{1+\kappa}}{\xi(v, u)\log^{1-\kappa} y\log^{1+\kappa} x}.
\end{equation}

In the remaining case $z<y$, we note that $z=x^{1/v}\geq X_1$
and a similar (but simpler) argument to the first case gives
\begin{equation}\label{eq35}
\begin{split}
\sum_{r\leq z}\frac{\rho(r)\log{r}}{r}
\leq & \sum_{r\leq z}\frac{\log{r}}{r}+\sum_{\substack{r\leq z,\\ r\equiv 1\pmod{l}}}\frac{(l-1)\log{r}}{r}<B_0\log z
\end{split}
\end{equation}
and
\begin{equation}\label{eq33b}
\begin{split}
V(P(z))\leq & \prod_{r<z}\left(1-\frac{1}{r}\right)\prod_{r<z, r\in U}\left(1-\frac{1}{r}\right)^{l-1} \\
< & \frac{e^{-\gamma}\log X_1}{\log^2 z}\left(1+\frac{1}{2\log^2 z}\right) \exp\left(\frac{s_0(l-1)}{X_1-1}+\frac{A_3}{\log^2 X_1}\right) \\
< & \frac{A_4\log X_1}{\log^2 z}.
\end{split}
\end{equation}
By (\ref{eq35}), we have $B(z)\leq B_0<v/u$ and therefore, as in the first case,
(\ref{eq33b}) gives
\begin{equation}\label{eq34b}
\frac{V(P(z))}{\psi_0(v, u)}\leq \frac{A_4\log X_1 (1+x^{2/u-1}) v^2}{\xi(v, u)\log^2 x}.
\end{equation}

Now, with the aid of inequalities (\ref{eq34a}) and (\ref{eq34b}), the lemma easily follows from (\ref{eq31}).
\end{proof}

Now we shall prove Theorem \ref{thm1}.
Let $q_0$ be the smallest prime factor of $N$ and
assume that $q_0\geq X_1(s_0; l, P)^v$ for any prime $l$ dividing $\mathcal{P}$ and $q_0\geq \max\{ 2(d+1)s, L(\ep, n)\}$.

Since $\prod_{i=1}^{s}h(p_i^{\alpha_i})\leq (q_0/(q_0-1))^s$, we obtain
\begin{equation}\label{eq37}
\prod_{j=1}^{t}h(q_j^{2\beta_j})\geq \frac{n}{d}\times\left(\frac{2(d+1)s-1}{2(d+1)s}\right)^s>\sqrt\frac{n}{d}.
\end{equation}

Let $d_l=\prod_{q} q/(q-1)$, where $q$ runs over all primes in $Q_l$.
It follows from (\ref{eq37}) that $\prod_{l\in \mathcal{P}}d_l\geq \sqrt{n/d}$.
Hence, we have that $d_l\geq\delta_1=\left(\frac{n}{d}\right)^{1/2\abs{\mathcal{P}}}$
for some $l$ in $\mathcal{P}$.

Recall that $\kappa=(l-1)/\vph(P)$.  Since $N$ has no prime factor less than $q_0$,
Lemma \ref{lm3} gives that
\begin{equation}\label{eq38}
\begin{split}
\log\delta_1\leq & \sum_{p\geq X_1^v, p\in \mathcal{P}}\frac{1}{p}\leq\int_{q_0}^{\infty}\frac{\pi_l(t)}{t^2}dt \\
< & \frac{\ep}{\log q_0}+\int_{q_0}^{q_0^v}\frac{B_1v^2(1+t^{2/u-1})}{\xi(v, u)t\log^2 t}dt+\int_{q_0^v}^{\infty}\frac{B_1v^{1+\kappa}(1+t^{2/u-1})}{\xi(v, u)t\log^{1+\kappa} t\log^{1-\kappa} q_0}dt \\
< & \frac{\ep}{\log q_0}+\frac{B_1(1+q_0^{2/u-1})}{\xi(v, u)\log q_0}\left(v^2\left(1-\frac{1}{v}\right)+\frac{v}{\kappa}\right) \\
< & \frac{\ep}{\log q_0}+\frac{\log X_1}{\log q_0}\left(\frac{B_2}{\kappa}+B_3\right)
\end{split}
\end{equation}
for some constants $B_2$ and $B_3$.
Hence, we have
\begin{equation}
\log q_0<\ep +\frac{\log X_1}{\log\delta_1}\left(\frac{B_2}{\kappa}+B_3\right)=\ep +\frac{2\log X_1\left(\frac{B_2}{\kappa}+B_3\right)}{(l-1)\log\frac{n}{d}}.
\end{equation}
In the next section, we shall show that we can take $X_1=x_1, B_2=17.62196$ and $B_3=129.5214$,
which proves Theorem \ref{thm1}.

\section{Distribution of primes in arithmetic progressions}

In order to complete the proof of Theorem \ref{thm1},
we must know some explicit estimates for the sum $\sum_p (\log p)/p$ and the product $\prod_p (1-1/p)$
with $p$ running over primes in an arithmetic progression.

We begin by introducing Chebyshev prime-counting functions for arithmetic progressions:
\begin{equation}
\psi(x; k, a)=\sum_{n\leq x, n\equiv a\pmod{k}}\Lambda(n)
\end{equation}
\begin{equation}
\theta(x; k, a)=\sum_{p\leq x, p\equiv a\pmod{k}}\log p.
\end{equation}

It is well-known that, for any modulus $k\leq \log x$ and congruence class $a\pmod{k}$ with $\gcd(a, k)=1$,
$\psi(x; k, a)$ is asymptotic to $x/\vph(k)$ with an error term $O(x/\vph(k)\log x)$.
Namely, we have
\begin{equation}\label{eq40a}
\abs{\psi(x; k, a)-\frac{x}{\vph(k)}}\leq \frac{A_0 x}{\vph(k)\log x}
\end{equation}
for $x\geq x_0$ with $x_0$ sufficiently large, where $A_0$ denotes some constant.
Indeed, we shall show the following explicit estimate.
\begin{lem}\label{lm41}
If $x\geq\exp(\exp(9))$ is a real number and $k$ is a positive integer coprime to $a$
not exceeding $\log x$, then
\begin{equation}\label{eq40b}
\abs{\psi(x; k, a)-\frac{x}{\vph(k)}}\leq \frac{0.00009x}{\vph(k)\log^2 x}.
\end{equation}
In other words, putting $A_0=0.00009$ and $x_0=\max\{\exp k, \exp(\exp(9))\}$,
the inequality (\ref{eq40a}) holds for $x\geq x_0$.
\end{lem}

\begin{proof}
For $k\geq 10^5$, Theorem 1.2 of \cite{BMOBR} gives
\begin{equation}
\abs{\psi(x, k, a)-\frac{x}{\vph(k)}}\leq \frac{1.012x^{1-40/(\sqrt{k}\log^2 k)}}{\vph(k)}+1.4579x\sqrt{X}\exp(-X),
\end{equation}
where $X=\sqrt{\log x/9.645908801}$.
Hence, we have
\begin{equation}
\abs{\psi(x, k, a)-\frac{x}{\vph(k)}}\leq \frac{10^{-30}x}{\vph(k)\log^2 x}
\end{equation}
for $x\geq e^k$ and $k\geq 10^5$.
Similar estimates have also been given in the author's preprint \cite{Ymd3} and another estimate is implicit in \cite{ChW}.

For $k<10^5$, we know from \cite{Pla} that no Dirichlet $L$-function modulo $k$
has a zero $s=\sigma+it$ with $\sigma>1/2$ and $\abs{t}\leq 1000$.
Now, putting $C_1(\chi, 1000)=9.14$,
we can confirm the conditions in Theorem 5 of \cite{Dus} for $x\geq x_0$.
Hence, we apply this theorem to obtain
\begin{equation}
\abs{\psi(x, k, a)-\frac{x}{\vph(k)}}\leq 3x\sqrt{\frac{kX}{9.14\vph(k)}}\exp(-X)<\frac{0.00009x}{\log^2 x}.
\end{equation}
Thus the lemma is proved.
\end{proof}

Based on this inequality, we shall prove the following estimates.
\begin{lem}\label{lm42}
Let $w$ and $z$ be arbitrary real numbers with $z\geq w\geq x_0$. 
Then the inequalities
\begin{equation}\label{eq43a}
\sum_{w<p\leq z, p\equiv a\pmod{k}}\frac{\log{p}}{p}
<\frac{1}{\vph(k)}\left(\log\frac{z}{w}+\frac{10^{-4}}{\log^2 w}+\frac{10^{-4}}{\log^2 z}+\frac{10^{-4}}{\log w}\right)
\end{equation}
and
\begin{equation}\label{eq44}
\prod_{w\leq p<z, p\equiv a\pmod{k}}\left(1-\frac{1}{p}\right)<\left(\frac{\log w}{\log z}\right)^{1/\vph(k)}\exp\left(\frac{1}{4000\vph(k)\log^2 x_0}\right).
\end{equation}
hold.

Moreover, if $z\geq x_0^{100.7}$, then we have
\begin{equation}\label{eq43b}
\sum_{p\leq z, p\equiv a\pmod{k}}\frac{\log p}{p}<\frac{\log z}{\vph(k)}+1.007\log x_0.
\end{equation}
\end{lem}

\begin{proof}
We begin by noting that Lemma \ref{lm41} yields
\begin{equation}\label{eq45}
\abs{\theta(x, k, a)-\frac{x}{\vph(k)}}<\frac{10^{-4}}{\vph(k)\log^2 x}
\end{equation}
for $x\geq x_0$.

Now we shall prove (\ref{eq44}).
By partial summation, we have
\begin{equation}
\begin{split}
& \sum_{\substack{w<p<z,\\ p\equiv a\pmod{k}}}\frac{1}{p}>\frac{\log\frac{\log z}{\log w}}{\vph(k)}-\frac{10^{-4}}{\vph(k)}\left(\frac{1}{\log^2 z}+\frac{1}{\log^2 w}+\int_w^z \frac{(1+\log t)dt}{t\log^4 t}\right)\\
& \quad > \frac{\log\frac{\log z}{\log w}}{\vph(k)}-\frac{10^{-4}}{\vph(k)}\left(\frac{3}{2\log^2 w}+\frac{1}{\log^2 z}+\frac{1}{3\log^3 w}-\frac{1}{3\log^3 z}\right) \\
& \quad > \frac{\log\frac{\log z}{\log w}}{\vph(k)}-\frac{1}{4000\vph(k)\log^2 w}
\end{split}
\end{equation}
for $z>w\geq x_0$ and therefore
\begin{equation}
\begin{split}
\prod_{\substack{w\leq p<z,\\ p\equiv a\pmod{k}}}\left(1-\frac{1}{p}\right)^{-1}
= & \exp\sum_{\substack{w\leq p<z,\\ p\equiv a\pmod{k}}}\left(\frac{1}{p}+\frac{1}{2p^2}+\cdots \right) \\
> & \exp\sum_{\substack{w\leq p<z,\\ p\equiv a\pmod{k}}}\frac{1}{p} \\
> & \left(\frac{\log z}{\log w}\right)^{1/\vph(k)}\exp\left(-\frac{1}{4000\vph(k)\log^2 w}\right)
\end{split}
\end{equation}
for $z\geq w\geq x_0$, which gives (\ref{eq44}).

Nextly, we shall prove (\ref{eq43b}).
Partial summation similar to above gives
\begin{equation}\label{eq46}
\begin{split}
\sum_{\substack{p\leq z,\\ p\equiv a\pmod{k}}}\frac{\log p}{p}\leq & \frac{\log (k+1)}{k+1}+\frac{\theta(z; k, a)}{z}+\int_{2k+1}^z\frac{\theta(t; k, a)}{t^2}dt \\
< & \frac{1}{\vph(k)}\left(\log (k+1)+1+\frac{10^{-4}}{\log^2 z}\right)+\int_{2k+1}^z \frac{\theta(t; k, a)}{t^2}dt.
\end{split}
\end{equation}

Recall that $\log x_0=\max\{k, e^9\}$.
Using the Brun-Titchmarsh theorem given in \cite{MV}, we have
\begin{equation}
\begin{split}
\int_{2k+1}^{x_0}\frac{\theta(t; k, a)}{t^2}dt<& \frac{2}{\vph(k)}\int_{2k}^{x_0}\frac{\log t dt}{t\log\frac{t}{k}} \\
= & \frac{2}{\vph(k)}\left(\log\frac{x_0}{2k}+\log k\left(\log\frac{\log x_0}{\log k}-\log\log 2\right)\right) \\
< & \frac{2.007}{\vph(k)}\log x_0.
\end{split}
\end{equation}
We can easily see that (\ref{eq45}) gives
\begin{equation}
\int_{x_0}^{z}\frac{\theta(t; k, a)}{t^2}dt<\frac{1}{\vph(k)}\left(\log\frac{z}{x_0}+\frac{10^{-4}}{\log x_0}\right).
\end{equation}
Inserting these upper bounds into (\ref{eq46}) yields
\begin{equation}
\sum_{\substack{p\leq z,\\ p\equiv a\pmod{k}}}\frac{\log p}{p}<\frac{1}{\vph(k)}\left(\log z+1.007\log x_0\right)
\end{equation}
for $z\geq x_0^{100.7}$, giving (\ref{eq43b}).

Finally, (\ref{eq43a}) immediately follows by using the partial summation
\begin{equation}
\sum_{\substack{w<p\leq z,\\ p\equiv a\pmod{k}}}\frac{\log p}{p}=\frac{\theta(z; k, a)}{z}-\frac{\theta(w; k, a)}{w}+\int_w^z\frac{\theta(z; k, a)}{t^2}dt
\end{equation}
and (\ref{eq45}).
\end{proof}

Now we shall complete the proof of Theorem \ref{thm1}.
In Lemma \ref{lm3}, we shall take $X_1=x_1(s_0; l, P)$ for each $l$ in $\mathcal{P}$.
In the case $z\geq y$, since $z\geq y\geq x_1(l, P)\geq \max\{x_0(P), x_0(l)^{100.7}\}$, (\ref{eq43b}) of Lemma \ref{lm42} applied with $k=l$ allows us to take $A_1=1.007\log x_0<0.01\log z$ in (\ref{eq32a})
and (\ref{eq43a}) of Lemma \ref{lm42} with $k=P$ allows us to take $A_2=1.1\times 10^{04}$ in (\ref{eq32b}).
Hence, in the case $z\geq y$, we can take $B_0=2.01$ in (\ref{eq32c}).
Since $z\geq X_1=x_1\geq x_0^{100.7}$, also in the case $z<y$, we can take $B_0=2.01$ in (\ref{eq35}).
In our setting of $x_1$, we can take $A_3=1/4000$ in (\ref{eq32d}) with $k=l$ and $k=P$ from (\ref{eq44}) of Lemma \ref{lm42},
noting that $x_1\geq x_0\geq \exp P$.
Since $x_1\geq 10s_0(l-1)+1$, we can take $A_4=\exp (0.1+10^{-9}-\gamma)$ and $B_1=e^{0.1+10^{-8}-\gamma}\log x_1$.

We choose $u=2+10^{-7}, v=8.35>4.03>B_0 u$ and assume that
$q_0\geq x_1(l)^v$ for any $l$ in $\mathcal{P}$ and $q_0\geq \max\{ 2(d+1)s, L(\ep, n)\}$.
Then the most right hand side of (\ref{eq38}) is at most
\begin{equation}
\frac{\ep}{\log q_0}+\frac{\log x_1}{\log q_0}\left(\frac{8.81098}{\kappa}+64.7607\right).
\end{equation}
Hence, we obtain
\begin{equation}
\log q_0< \ep +\frac{(17.62196\vph(P)+129.5214(l-1))\abs{\mathcal{P}}\log x_1}{(l-1)\log\frac{n}{d}}.
\end{equation}
This implies that $q_0\leq C_0$ and the proof of Theorem \ref{thm1} is complete.

\section[Proof of Theorems 1.2-1.4]{Proof of Theorems \ref{thm2}-\ref{thm4}}

Firstly, we shall prove Theorem \ref{thm4}.
Assume that $N=p_1^{2\alpha_1}p_2^{2\alpha_2}\cdots p_t^{2\alpha_t}$ satisfies that
$\sigma(N)\geq 2N$ has no prime factor congruent to $5$ or $7$ modulo $8$
and each $2\alpha_i+1$ is divisible by some prime in $\mathcal{P}$.

For each $l\in\mathcal{P}$, let $R_l$ denote the set of primes $p_i$ with $2\alpha_i+1$ divisible by $l$.
By assumption, any $p_i$ belongs to $R_l$ for some $l$ in $\mathcal{P}$.
Let $a_1, a_2\pmod{8l}$ be the congruence classes that are
congruent to $1\pmod{l}$ and $5, 7\pmod{8}$ respectively.
If $p\in R_l$, then $p^{l-1}+p^{l-2}+\cdots +1$ has no prime factor congruent to $a_1$ or $a_2 \pmod{8l}$.

We shall show that $\prod_{p\geq C_1, p\in R_l}\frac{p}{p-1}<2^{1/\abs{\mathcal{P}}}$ for all $l$ in $\mathcal{P}$,
which would imply that $N$ must have some prime factor smaller than  $C_1$
in order to satisfy $\sigma(N)/N\geq 2$.
But, in order to prove Theorem \ref{thm4}, we shall apply our sieve argument setting
$\Omega^{(l)}_p=\{n\mid n(n^{l-1}+n^{l-2}+\cdots +1)\equiv 0\pmod{p}\}$
for primes $p$ congruent to $a_1$ or $a_2 \pmod{8l}$
and $\Omega^{(l)}_p=\{n\mid n\equiv 0\pmod{p}\}$ for other primes.

Let $\pi^\prime_l(x)$ denote the number of primes $\leq x$ that belong to $R_l$.
We have
\begin{equation}
\pi^\prime_l(x)<S(A, y, \Omega^{(l)})+y
\end{equation}
for any $y$.
Let $z$ be an arbitrary real number $\geq x_4=x_3^{100.7}$.
Then, observing that $x_3\geq x_0(8l)$, (\ref{eq43b}) gives
\begin{equation}
\sum_{\substack{p\leq z,\\ p\equiv a_1, a_2\pmod{8l}}}\frac{\log p}{p}<\frac{1.01}{2\vph(l)}\log z
\end{equation}
and therefore
\begin{equation}\label{eq51}
\begin{split}
\sum_{r\leq z}\frac{\rho(r)\log{r}}{r}
\leq & \sum_{r\leq z}\frac{\log{r}}{r}+\sum_{\substack{r\leq z,\\ r\equiv a_1, a_2\pmod{8l}}}\frac{(l-1)\log{r}}{r}<1.505\log z.
\end{split}
\end{equation}
Observing that $z\geq x_4>\exp(\exp(9))$ and using (\ref{eq44}), we have
\begin{equation}\label{eq52}
\begin{split}
V(P(z))\leq & \prod_{r<z}\left(1-\frac{1}{r}\right)\prod_{\substack{x_3\leq r<z,\\ r\equiv a_1, a_2\pmod{8l}}}\left(1-\frac{1}{r}\right)^{l-1} \\
< & \frac{e^{-\gamma}\log^{1/2} x_3}{\log^{3/2} z}\left(1+\frac{1}{\log z}\right) \exp \left(\frac{1}{4000\log^2 x_3}\right) \\
< & \frac{0.56146 \log^{1/2} x_3}{\log^{3/2} z}.
\end{split}
\end{equation}

From (\ref{eq51}) and (\ref{eq52}), the sieve inequality given in Lemma \ref{lm22}
with $B=1.505, u=2.000007$ and $v=7.538$ gives that, if $x\geq x_4^{7.538}$, then
\begin{equation}
S(A, x^{1/u}, \Omega^{(l)})\leq \frac{16.65708x\log^{1/2}x_3}{\log^{3/2} x}
\end{equation}
and therefore
\begin{equation}
\pi^\prime_l (x)\leq S(A, x^{1/u}, \Omega^{(l)})+x^{1/u}<\frac{16.65709x\log^{1/2}x_3}{\log^{3/2} x}.
\end{equation}

Since $C_1=x_3^{2310\abs{\mathcal{P}}^2}>x_4^{7.538}$, we have, as in the proof of Theorem \ref{thm1},
\begin{equation}
\prod_{p\geq C_1, p\in R_l}\frac{p}{p-1}<\exp\left(\frac{2}{C_1}+\frac{33.31418\log^{1/2}x_3}{\log^{1/2} C_1}\right)<2^{1/\abs{\mathcal{P}}}
\end{equation}
for each prime $l\in\mathcal{P}$, which proves Theorem \ref{thm4}.

Now, all that remains is to derive Theorems \ref{thm2} and \ref{thm3} from Theorem \ref{thm4}.
If $N=p_1^{2\beta_1}p_2^{2\beta_2}\cdots p_t^{2\beta_t}$ satisfies $\sigma(N)=2N+1$
and each $2\beta_i+1$ is divisible by some prime in $\mathcal{P}$, then, as mentioned above,
Cattaneo has shown that $\sigma(N)$ has no prime factor congruent to $5$ or $7$ modulo $8$
and therefore $N$ satisfies the hypothesis of Theorem \ref{thm4}.  Hence,
$N$ must have a prime factor smaller than $C_1$.  This proves Theorem \ref{thm2}.

If $m$ is even, $n$ is odd and $m, n$ are relatively prime integers satisfying
$\sigma(m)\sigma(n)=(m+n)^2$ and $mn=2^\alpha p_1^{2\beta_1}p_2^{2\beta_2}\cdots p_t^{2\beta_t}$,
then we see that $m=2A^2$ and $n=B^2$ for some odd integers $A, B$ from Kishore \cite{Kis}.
Since $m, n$ are relatively prime, so are $A, B$.  Hence, $\sigma(m)\sigma(n)=(m+n)^2=(2A^2+B^2)^2$
has no prime factor congruent to $5$ or $7$ modulo $8$.
Now, taking $N=mn/2=A^2B^2=p_1^{2\beta_1}p_2^{2\beta_2}\cdots p_t^{2\beta_t}$,
we see that $\sigma(N)/N>\sigma(mn)/(2mn)=\sigma(m)\sigma(n)/(2mn)=(m+n)^2/(2mn)>2$
and $\sigma(N)=\sigma(m)\sigma(n)/3=(2A^2+B^2)^2/3$ has no prime factor congruent to $5$ or $7$ modulo $8$.
So that, Theorem \ref{thm3} follows from Theorem \ref{thm4}.

{}
\vskip 12pt

{\small Tomohiro Yamada}\\
{\small Center for Japanese language and culture\\Osaka University\\562-8678\\3-5-10, Sembahigashi, Minoo, Osaka\\Japan}\\
{\small e-mail: \protect\normalfont\ttfamily{tyamada1093@gmail.com}}

\begin{thebibliography}{}
\bibitem{BMOBR}
Michael A. Bennett, Greg Martin, Kevin O'Bryant and Andrew Rechnitzer,
{\em Explicit bounds for primes in arithmetic progressions},
Illinois J. Math. \textbf{62} (2018), 427--532.

\bibitem{Cat}
Paolo Cattaneo, {\em Sui numeri quasiperfetti},
Boll. Un. Mat. Ital. (3) \textbf{6} (1951), 59--62.

\bibitem{Co1}
Graeme L. Cohen, {\em The nonexistence of quasiperfect numbers of certain forms},
Fibonacci Quart. \textbf{20} (1982), 81--84.

\bibitem{Co2}
G. L. Cohen, {\em On the largest component of an odd perfect number}, J. Austral. Math. Soc. Ser. A \textbf{42} (1987), 280--286.

\bibitem{ChW}
Chen, Jing-Run and Wang, Tian-Ze, {\em On the odd Goldbach problem},
Acta Math. Sinica \textbf{39} (1996), 169--174 (in Chinese).

\bibitem{CoW}
G. L. Cohen and R. J. Williams, {\em Extensions of some results concerning odd perfect numbers}, Fibonacci Quart. \textbf{23} (1985), 70--76.

\bibitem{Dus}
P. Dusart, {\em Estimates of $\theta(x; k, l)$ for large values of $x$},
Math. Comp. \textbf{71} (2001), 1137--1168.

\bibitem{EP}
Ronald Evans and Jonathan Pearlman, {\em Nonexistence of odd perfect numbers of a certain form}, Fibonacci Quart. \textbf{45} (2007), 122--127.

\bibitem{FNO}
S. Adam Fletcher, Pace P. Nielsen and Pascal Ochem,
{\em Sieve methods for odd perfect numbers},
Math. Comp. \textbf{81} (2012), 1753--1776.

\bibitem{Gre}
G. Greaves, {\em Sieves in Number Theory}, Springer-Verlag, Berlin, 2001.

\bibitem{Gru}
O. Gr{\"u}n, {\em {\"U}ber ungerade vollkommene Zahlen},
Math. Z., \textbf{55} (1952) 353--354.

\bibitem{Ha1}
Peter Hagis Jr., {\em Relatively Prime Amicable Numbers of Opposite Parity}, Math. Mag. \textbf{43} (1970), 14--20.

\bibitem{Ha2}
Peter Hagis Jr., {\em Lower Bounds for Relatively Prime Amicable Numbers of Opposite Parity}, Math. Comp. \textbf{24} (1970), 963--968.

\bibitem{Ha3}
Peter Hagis Jr., {\em On the number of prime factors of a pair of relatively prime amicable numbers},
Math. Mag. \textbf{48} (1975), 263--266.

\bibitem{HC}
Peter Hagis, Jr. and Graeme L. Cohen, {\em Some results concerning quasiperfect numbers},
J. Aust. Math. Soc. (Ser. A) \textbf{33} (1982), 275--286.

\bibitem{HM}
P. Hagis Jr. and Wayne L. McDaniel, {\em A new result concerning the structure of odd perfect numbers}, Proc. Amer. Math. Soc. \textbf{32} (1972), 13--15.

\bibitem{IK}
H. Iwaniec and Kowalski, {\em Analytic Number Theory}, American Mathematical Society, Providence, RI, 2004.

\bibitem{Ka1}
H.-J. Kanold, {\em Untersuchungen {\"u}ber ungerade vollkommene Zahlen}, J. Reine Angew. Math. \textbf{183} (1941), 98--109.

\bibitem{Ka2}
Hand-Joachim Kanold, {\em {\"U}ber befreundete Zahlen}, II, Math. Nachr. \textbf{10} (1953), 99--111.

\bibitem{Ka3}
Hand-Joachim Kanold, {\em Untere Schranken f{\"u}r teilerfremde befreundete Zahlen}, Arch. Math. (Basel) \textbf{4} (1953), 399--401.

\bibitem{Kis}
Masao Kishore, {\em On the equation $\sigma(m)\sigma(n)=(m+n)^2$}, Fibonacci Quart. \textbf{19}, 21--23.

\bibitem{Mc1}
W. L. McDaniel, {\em The non-existence of odd perfect numbers of a certain form}, Arch. Math. (Basel) \textbf{21} (1970), 52--53.

\bibitem{Mc2}
W. L. McDaniel, {\em On the divisibility of an odd perfect number by the sixth power of a prime}, Math. Comp. \textbf{25} (1971), 383--385.

\bibitem{MDH}
W. L. McDaniel and P. Hagis Jr., {\em Some results concerning the non-existence of odd perfect numbers of the form $p^{a}M^{2\beta }$}, Fibonacci Quart. \textbf{13} (1975), 25--28.

\bibitem{MV}
H. L. Montgomery and R. C. Vaughan, {\em The large sieve}, Mathematika \textbf{20} (1973), 119--134.

\bibitem{Pla}
David J. Platt, Numerical computations concerning the GRH,
{\it Math. Comp.} \textbf{85} (2016), 3009--3027.

\bibitem{RS}
J. B. Rosser and L. Schoenfeld, {\em Approximate formulas for some functions of prime numbers}, Illinois J. Math.\textbf{6} (1962), 64--94.

\bibitem{St}
R. Steuerwald, {\em Versch{\"a}rfung einen notwendigen Bedingung f{\"u}r die Existenz einen ungeraden vollkommenen Zahl}, S.-B. Bayer. Akad. Wiss. 1937, 69--72.

\bibitem{Ymd1}
Tomohiro Yamada, {\em Odd perfect numbers of a special form}, Colloq. Math. \textbf{103}, 303--307.

\bibitem{Ymd2}
Tomohiro Yamada, {\em On the divisibility of odd perfect numbers by a high power of a prime},
preprint, \url{https://arxiv.org/abs/math/0511410}.

\bibitem{Ymd3}
Tomohiro Yamada, Explicit formulae for primes in arithmetic progressions I,
preprint, \url{https://arxiv.org/abs/1306.5322}.

\bibitem{Ymd4}
Tomohiro Yamada, {\em Quasiperfect numbers with the same exponent},
Integers \textbf{19} (2019), \#A35.

\end{thebibliography}
\end{document}